\documentclass[12pt,leqno]{article}
\usepackage{amsmath}
\usepackage{subfigure}
\usepackage{amsthm}
\usepackage{amstext}
\usepackage{amsopn}
\usepackage{texdraw}
\usepackage{graphicx}
\usepackage{pdfpages}
\usepackage{multicol,graphicx,color}
\usepackage{pslatex}
\usepackage{amsthm}
\usepackage{amsmath}
\usepackage{amssymb}
\usepackage{latexsym}
\usepackage{lscape}
\usepackage{epsfig}
\usepackage{pstricks}
\usepackage{amsfonts}
\usepackage{enumerate}
\usepackage{exscale}
\usepackage{relsize}
\usepackage{multirow}
\usepackage{mathrsfs}
\usepackage{tikz}
\usetikzlibrary{shapes.geometric, arrows}

\newtheorem{theorem}{Theorem}[section]
\newtheorem{lemma}[theorem]{Lemma}

\theoremstyle{definition}

\newtheorem{example}[theorem]{Example}

\theoremstyle{remark}
\newtheorem{remark}[theorem]{Remark}
\usepackage{tikz}
\usetikzlibrary{shapes.geometric, arrows}

\tikzstyle{startstop} = [rectangle, rounded corners, 
minimum width=5cm, 
minimum height=1cm,
text centered, 
text width=4.5cm, 
draw=black, 
fill=red!30]

\tikzstyle{io} = [trapezium, 
trapezium stretches=true, 
trapezium left angle=70, 
trapezium right angle=110, 
minimum width=5cm, 
minimum height=1cm, text centered, 
draw=black, fill=blue!30]

\tikzstyle{process} = [rectangle, 
minimum width=5cm, 
minimum height=1cm, 
text centered, 
text width=4.5cm, 
draw=black, 
fill=orange!30]

\tikzstyle{decision} = [diamond, 
minimum width=3cm, 
minimum height=1cm, 
text centered, 
text width=4.5cm, 
draw=black, 
fill=green!30]
\tikzstyle{arrow} = [thick,->,>=stealth]

\begin{document}
\title{On the Uniqueness of Convex Central Configurations \\
in the Planar $4$-Body Problem}
\author{Shanzhong Sun\\
Department of Mathematics \\
and\\
Academy for Multidisciplinary Studies\\
Capital Normal University, Beijing 100048, P. R. China\\
sunsz@cnu.edu.cn\\
 Zhifu Xie\\
 School of Mathematics and Natural Science\\
 The University of Southern Mississippi\\
 Hattiesburg, MS 39406,USA\\
 Zhifu.Xie@usm.edu\\
 Peng You\\
  School of Mathematics and Statistics\\
  Hebei University of Economics and Business\\
  Shijiazhuang Hebei 050061, P. R. China\\
you-peng@163.com
  }
 \date{}
\maketitle
\begin{center}
{Dédier à Alain Chenciner avec admiration et amitié}
\end{center}

\begin{abstract} 
 In this paper, we provide a rigorous computer-assisted proof (CAP) of the conjecture that in the planar four-body problem there exists a unique convex central configuration for any four fixed positive masses in a given order belonging to a closed domain in the mass space. The proof employs the Krawczyk operator and the implicit function theorem (IFT). Notably, we demonstrate that the implicit function theorem can be combined with interval analysis, enabling us to estimate the size of the region where the implicit function exists and extend our findings from one mass point to its neighborhood.

 \end{abstract}

 \section{Introduction}
 The planar Newtonian $N$-body problem describes the motion of $N$ celestial bodies considered as point particles in the plane under the universal gravitational law, which is governed by a system of second order ordinary differential equations:
 $$m_i\ddot{q_{i}} = \sum_{j=1,j\neq i}^{N} \frac{Gm_im_{j}(q_{j}-q_{i})}{\left\|q_j-q_i\right\|^3}, \hspace{0.5cm}  i = 1,2,..,N.$$
Hereafter the gravitational constant $G$ is set to 1 for convenience and $q_j\in\mathbb{R}^2 $ is the position vector of the $j$-th point mass $m_j$ in the planar inertial barycentric frame. $r_{ij}=\|q_i-q_j\|$ is the Euclidean distance between $q_i$ and $q_j$.  Let 
 \begin{equation} \label{NT2}
C=m_1q_1+\cdots+m_N q_N
\end{equation} 
be the linear moment and $M=m_1+m_2 + \cdots+m_N$ be the  total mass with $c=C/M $ the center of mass. 
Given a mass vector $m=(m_1, m_2, \cdots, m_N)$, a configuration $q=(q_1, q_2, \cdots, q_N)$ is called a {\it central configuration} for $m$ if there exists a positive constant $\lambda$ such that 
 \begin{equation}\label{NT3}
 \sum_{j=1,j\neq i}^{N} \frac{m_i m_{j}(q_{j}-q_{i})}{r_{ij}^3}=-\lambda{m_i(q_{i}-c)}.
\end{equation}
We say that two central configurations are equivalent if they can be related by rigid motions or scalings with respect to the center of mass. Later when we talk about the number of central configurations, we always mean their equivalence class.

There is another equivalent characterization of central configurations.
 Let $U$ be the negative Newtonian potential energy or the self potential given by 
\begin{equation}\label{Upot}
U=\sum_{1\leq i<j\leq N}^{N} \frac{m_i m_j}{r_{ij}},
\end{equation}
and let $I$ be the moment of inertia given by 
\begin{equation}\label{Imom}
I=\frac{1}{M} \sum_{1\leq i<j\leq N}^{N} m_i m_j r_{ij}^2.
\end{equation}
If the center of mass is fixed at the origin of the plane, the moment of inertia is equal to $I=\sum_{i} m_i  |q_i|^2$.  Due to the fact that the potential is homogeneous of degree $-1$, it is well known that $\lambda=\frac{U}{I}$ by the classical Euler Theorem. Now the central configurations are the critical points of $U$ restricted to the \lq\lq shape sphere\rq\rq\,\,  $I=1$, or equivalently they are the critical points of $IU^2$ without constraints. 


Central configurations are significant in understanding the solution structure and dynamical behavior of the $N$-body problem. One of their essential roles is to provide insights into the dynamics in the vicinity of collisions. Additionally, these configurations serve as the initial positions for homographic periodic solutions that preserve the configuration's shape. As a result, they find practical applications, for example, in spacecraft missions. For more on the background about central configurations, please refer to \cite{RM15} and the references therein.

The famous and long-standing Chazy-Wintner-Smale conjecture in the $N$-body problem says that for any given masses, there is only a finite number of central configurations (\cite{RM15, S98}). It has been confirmed for the planar $4$-body problem by Hampton and Moeckel (\cite{HM}) {for any masses} and for the  planar $5$-body problem by Albouy and Kaloshin (\cite{AK}) for any masses except for masses belonging to a codimension $2$ subvariety in the mass space defined by a polynomial system. In the general case, it is still widely open.

In the current paper we will focus on  central configurations in the planar 4-body problem. Though the finiteness of central configurations is already established and reconfirmed in \cite{AK, HM}, a precise finite number or better upper bound on the number of central configurations is still expected.


A related problem concerns the counting of special central configurations (e.g., configurations with interesting geometry) especially for those masses with some kind of \lq\lq symmetry\rq\rq\,\, (e.g., equal masses). Among them, the convex central configurations have attracted attention in many recent papers (\cite{A95,A96,AFS,CCLP, CCR18, CR, FLM, LS1, San18, San21a, San21b, Sch02, Xie2, Xie}, just to mention a few) which are also our main focus in the current paper. By (strict) convexity we mean that no body is contained inside or on the convex hull of the other three bodies. 

The study on convex central configurations dates back to the classical work \cite{MB} where MacMillan and Bartky proved that there exists a convex central configuration in the planar $4$-body problem for any given cyclic ordering of four positive masses. They also derived useful information on the admissible shapes of the $4$-body convex central configurations which was generalized and refined in \cite{YL} and will be helpful for us. Xia (\cite{Xia}) reconfirmed it in a simpler way by proving the stronger existence of local minima of $IU^2$. He also conjectured that for any $N(\geq 5)$ positive masses and any cyclic order of $N$ points on $S^1$, there is a convex planar central configuration of minimum type with the cyclic order.


However for the planar 4-body problem it was conjectured (\cite{ACS}, Problem 10) that there is only one convex central configuration for any four fixed positive masses in a given order. The conjecture is true for four equal masses and some other special settings with constraints on the shapes or masses as cited above, but there is no proof for the general case. It has resisted any attempts for a while due to the fact that the Hessian matrix of the Newton potential is hard to control.

Recently, Corbera-Cors-Roberts (\cite{CCR19}) classified the full set of convex central configurations in the planar 4-body problem by choosing convenient coordinates, and proved it to be three dimensional. As remarked by the authors, the injectivity of the mass map defined in their paper should be established in order to confirm the uniqueness.


We confirm the uniqueness conjecture  by a  rigorous computer-assisted-proof for almost all masses.  All symbolic and numerical computations were carried out by using MATLAB
(2017b),  the open-source software SageMath 9.6 (2022), and Maple 2022. More precisely, we prove the following theorem after normalizing the largest mass $m_2$ to be $1$.

\begin{theorem}\label{MainMainTheorem}
For the planar Newtonian 4-body problem there is a unique convex central configuration for any fixed positive masses $m=(m_1,1,m_3,m_4)$ such that $(m_1,m_3,m_4)\in [0.6,1]^3$ in a given order as in Figure \ref{CC3}.     
\end{theorem} 

\begin{remark}
\begin{itemize}
    \item The settings and conventions will be given in detail in \S 2.
    
    \item In the statement of the theorem, the number $0.6$ is far {from being} optimal. In fact, for any $(m_1,m_3,m_4)\in $ $[\delta_0,1]^3$ with given $\delta_0$ a fixed positive real number, our methods should work. We will address this point in the last section.
   
    \item The {remaining masses are those} near zero which we will also comment on the last section, and it seems to be an interesting and challenging problem both analytically and numerically.

    \item Note that in our setting we fix the relative ordering of the masses and vary the masses. This is equivalent to prove the theorem for a given mass in any ordering. 
    
\end{itemize}
\end{remark}


 The main idea of the proof can be summarized as follows: We initially validate the conjecture for a set of selected mass vectors $m_0$  using interval analysis and MacMillan-Bartky's classical result on the existence of convex central configurations in the domain of mass. Subsequently, by determining the implicit function theorem's domain size, we can extend the central configuration's existence and uniqueness to a neighborhood of $m_0$, say a  mass ball $B(m_0,\epsilon)$ for the position ball $B(x_0,r)$. Then we verify that no solutions to the central configuration equations exist outside the previously considered  position domain $B(x_0,r)$ for the mass ball $B(m_0,\epsilon)$. Lastly we proceed to establish the uniqueness within the stated closed domain $[0.6,1]^3$ by iteratively implementing the aforementioned method. We call such a mass ball $B(m_0, \epsilon)$ a uniqueness mass ball if there is a unique convex central configuration for any mass in this ball. In short, the idea of the proof is to construct enough  uniqueness mass balls to cover the whole mass space. Please refer to the flow chart in Figure \ref{FlowChart} for a visual representation of the process.

\begin{figure}
  \centering
  \resizebox{14cm}{!} 
  {
\begin{tikzpicture}[node distance=3.3cm]
\node (start1) [startstop] {Initialize masses: $m_2=1, $ $(m_1,m_3,m_4)\in [0.6,1]^3$
Initialize position $x=[x_1,y_1,x_3,y_3]\in \mathbf{U}$};
\node (pro1) [process, below of=start1] {Process1: \\Interval Arithmetic  \& Krawczyk Operator,\\ uniquness of CC for $m_0$};
\node (pro2) [process, below of=pro1,yshift=-0.6cm] {Process 2:\\Implicit Function Theorem,\\ uniqueness mass ball};
\node (pro3) [process, below of=pro2, yshift=0.5cm] {Process 3: \\Confirm that CCs for $B(m_0,\epsilon)$ are in $B(x_0,r)$ };
\node (dec1) [decision, right of=pro3, xshift=3.0cm] {Decision:\\ uniqueness mass balls cover the mass space };
\node (stop) [startstop, below of=dec1,yshift=-1.4cm] {Stop};
\node (pro4) [io, above of=dec1, yshift=1.5cm] {generate more uniqueness mass balls};
\draw [arrow] (start1) -- node[anchor=east] {$m_0$}(pro1);
\draw [arrow] (pro1) -- node[anchor=east] { $x_0$ and $m_0$}(pro2);
\draw [arrow] (pro2) -- node[anchor=east] { $B(m_0,\epsilon)$ and $B(x_0,r)$} (pro3);
\draw [arrow] (pro3) -- node[anchor=south]{}(dec1);
\draw [arrow] (dec1) -- node[anchor=east] {yes} (stop);
\draw [arrow] (dec1) -- node[anchor=west] {no} (pro4);
\draw [arrow] (pro4) |-(start1);
\end{tikzpicture}
}
\caption{Flow chart to illustrate the computation process. The goal of the computation is to generate enough uniqueness mass balls to cover the whole mass space. Processes 1, 2, and 3 correspond to sections 3, 4 and 5 respectively. }\label{FlowChart}
\end{figure}
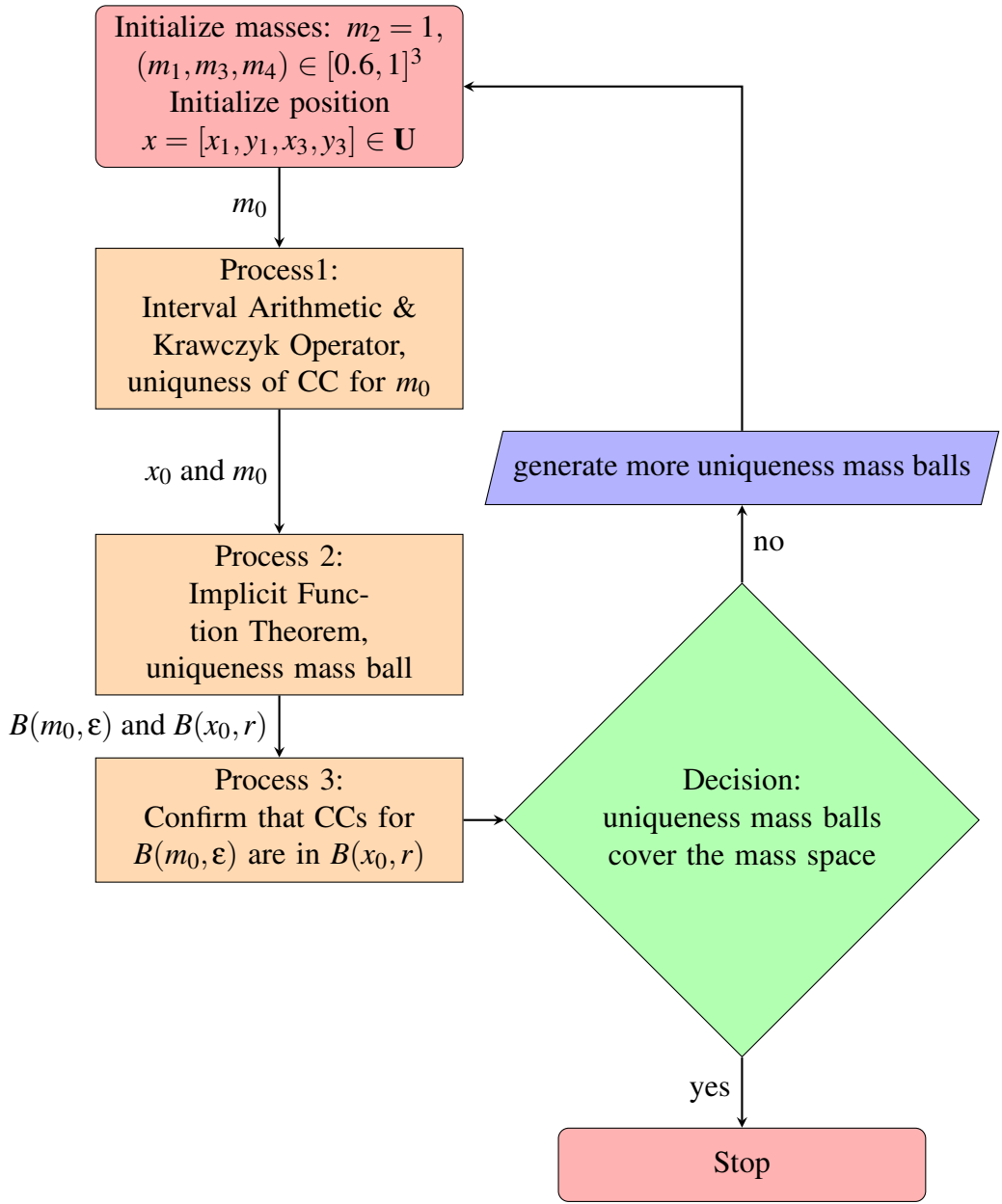

Note that Lee and Santoprete (\cite{LS09}) have already used the interval Krawczyk method to study the exact counting of the central configurations in the planar $5$-body problem with equal masses. Moczurad and Zgliczynski (\cite{MZ}) also used this method to study the central configurations for equal masses for planar $5, 6,7$-body problems. These papers are quite different from our situation here because we are interested in ALL masses.

After the introduction to the uniqueness conjecture and our main result establishing the conjecture with the idea of the proof in the current section, we give in \S 2 the form of the  equations about central configurations convenient for our purpose with useful results which add suitable constraints about the convex central configuration. We introduce the basics about the interval algorithm in \S 3, then based on this we prove the existence of a convex central configuration for a mass vector at the meshed grid point. \S 4 is about the size of the domain where the implicit function theorem is valid. We use the theorem from \cite{JCB}, and for completeness, we give a proof. We exclude the possibilities for the existence of convex central configurations outside of the domain generated by IFT. Finally in \S 5 we establish a computer-assisted-proof to our main Theorem \ref{MainMainTheorem}. We close the paper with some remarks and comments in \S 6.

As always insisted by Alain Chenciner, understanding the reasoning behind the truth is of primary importance, and we leave a more conceptual proof for the uniqueness to future research.

\section{Equations for the Convex Central Configurations in the Planar 4-Body Problem }
 For the planar $4$-body central configurations we can use the equivalent Dziobek-Laura-Andoyer equations ( \cite{YH}, p.241) 
\begin{equation}\label{NT4}
 f_{ij}=\sum_{k=1, k\neq i, j}^{4} m_{k}(R_{ik}-R_{jk})\Delta_{ijk} = 0
\end{equation} 
 for $1\leq i < j \leq 4$, where $R_{ij} = 1/r_{ij}^3$ and 
      $\Delta_{ijk} = (q_{i} - q_{k})\times (q_{j}-q_{k})$ with $\times $ the cross product of two vectors. Thus the area   $\Delta_{ijk}$ gives the twice the signed area of the triangle    $\Delta{ijk}$. If the bodies $m_i, m_j, m_k$ are in counter clockwise, the area of $\Delta_{ijk}$ is positive. Otherwise, it is negative.
      
 From the point of view of CAP (computer-assisted proofs) in the problem of finding and
counting all CCs, the dimension and the size of the search domain is fundamental. The obstacles  arise for the following reasons:
 \begin{itemize}
\item two or more bodies might be arbitrary close to each other (collisions, singularities of the system of equations);
\item  bodies might be arbitrary far from each other (unbounded domain);
\item  central configurations are not isolated due to the invariance of the system under rotation, scaling or reflection which could give some degeneracy in search algorithm.
 \end{itemize}

Numerically solving the central configuration equations \eqref{NT4} directly in a naive way is not feasible due to these issues. However, these problems can be addressed either by imposing conditions on the masses or by employing alternative methods. By demanding sufficiently large masses (e.g., $0.6$ as in Theorem \ref{MainMainTheorem}), the first issue can be avoided, and we will revisit this point in the future, as discussed in \S 6. By leveraging existing results, we can reduce the number of variables and limit the search domain to a bounded region in the configuration space, far from collisions. This approach allows us to isolate central configurations and eliminate continuous symmetries.
     
  \begin{lemma}[\cite{RM}, Perpendicular Bisector Theorem, p.510] Let $q=(q_1, q_2, \cdots, q_n)$ be a planar central configuration for $m\in (\mathbf{R^+})^n$ and let $q_i$ and $q_j$ be any two positions.  Denote by $Q_1, Q_2, Q_3,$ and $Q_4$ the four open quadrants in $\mathbf{R}^2$ in clockwise order obtained by deleting the line passing through $q_i$ and $q_j$ and the perpendicular bisector of the line segment $\overline{q_iq_j}$.  Then if there exists some $q_k\in Q_1 \cup Q_3$, there must exist at least another $q_l \in Q_2 \cup Q_4$. 
  \end{lemma}

   \begin{lemma}[\cite{YL}, Theorem 4.1] 
  For any $m=(m_1, m_2, m_3,m_4)\in (\mathbf{R}^+)^4$, let $q=(q_1, q_2,$ $ q_3,$ $ q_4)$ be a convex central configuration for $m$ as in Figure \ref{CC3}. Let $\overline{q_2q_4}$ be the line segment between points $q_2$ and $q_4$. Then $q_1$ and $q_3$ must satisfy the following inequalities: 
  \begin{equation}
  \frac{\sqrt{3}}{6} r_{24}< \max \{ |q_1 - \overline{q_2q_4} |, |q_3 - \overline{q_2q_4} | \} <\frac{\sqrt{3}}{2} r_{24},
  \end{equation}
 \begin{equation}
\left(1- \frac{\sqrt{3}}{2}\right) r_{24}< \min \{ |q_1 - \overline{q_2q_4} |, |q_3 - \overline{q_2q_4} | \} .
  \end{equation}
  \end{lemma}
    \begin{figure}[ht]
\centering
\includegraphics[width=0.4\linewidth,height=2.5in]{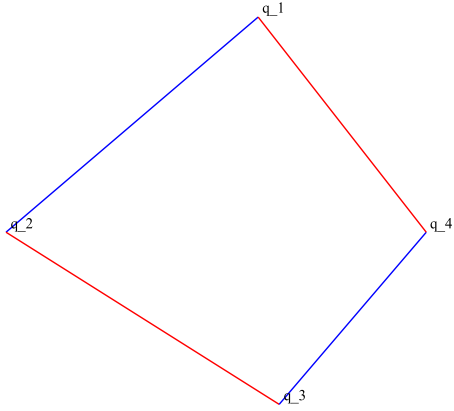}
 \caption{Admissible Convex Configuration: $q_1, q_2, q_3, q_4$ are in counter clockwise. $q_1=(x_1, y_1)$,  $q_2=(-1,0)$, $q_3=(x_3, y_3)$, and $q_4=(1, 0)$. $[x_1,y_1,x_3,y_3]\in \mathbf{U}.$  
 }\label{CC3}
\end{figure}

 Due to the invariance of central configurations under rotation and scaling, we can choose the coordinates for the $4$-body configuration as $q_1=(x_1, y_1)$,  $q_2=(-1,0)$, $q_3=(x_3, y_3)$, and $q_4=(1, 0)$. 
  From the above two lemmas, it is easy to prove the following theorem whose proof is omitted. 
  
  \begin{theorem}\label{Theorem1}
   For any $m=(m_1, m_2, m_3,m_4)\in (\mathbf{R}^+)^4$, let $q=(q_1, q_2, q_3, q_4)$ be a convex central configuration for $m$.  Without loss of generality, we assume $q_1=(x_1, y_1)$,  $q_2=(-1,0)$, $q_3=(x_3, y_3)$, and $q_4=(1, 0)$ as given in Figure \ref{CC3} after rotation, translation and scaling. Then
   \begin{equation}
   -1<x_1<1, -1<x_3<1, x_1x_3\geq 0,
   \end{equation}
   \begin{equation}
  2\left(1- \frac{\sqrt{3}}{2}\right) <y_1<{\sqrt{3}}, \hspace{0.6cm}  -{\sqrt{3}}<y_3< -2\left(1- \frac{\sqrt{3}}{2}\right),
  \end{equation}
  \begin{equation}
  \sqrt{(x_1+1)^2+y_1^2}<2,   \hspace{0.6cm}  \sqrt{(x_1-1)^2+y_1^2}<2, 
  \end{equation}
  \begin{equation}
  \sqrt{(x_3+1)^2+y_3^2}<2, \hspace{0.6cm}  \sqrt{(x_3-1)^2+y_3^2}<2.
  \end{equation}
Denote by $\mathbf{U}$ the admissible region of the position vectors which satisfy the above conditions.\\
If $m_2> m_4$, then $0<x_1<1$ and $0<x_3<1$. If $m_2=m_4$, then $x_1=x_3=0$.  \\

  \end{theorem}

 For the $4$-body problem, the equation \eqref{NT4} gives  a system of six equations. Here are the expressions of the equations: 
  $$  \Delta_{123}=\left({x_3} +1\right) {y_1} -{y_3} \left({x_1} +1\right),  \Delta_{124}=2y_1,$$ 
  $$ \Delta_{134}=\left(-{x_3} +1\right) {y_1} +{y_3} \left({x_1} -1\right),\Delta_{234}=-2y_3.$$
  \begin{equation}
  \begin{array}{ll}
  f_{12}=& {m_3} \left(\frac{1}{\left(\left({x_1} -{x_3} \right)^{2}+\left({y_1} -{y_3} \right)^{2}\right)^{\frac{3}{2}}}-\frac{1}{\left(\left({x_3} +1\right)^{2}+{y_3}^{2}\right)^{\frac{3}{2}}}\right) \left(\left({x_3} +1\right) {y_1} -{y_3} \left({x_1} +1\right)\right)+\\
 &  2 {m_4} \left(\frac{1}{\left(\left({x_1} -1\right)^{2}+{y_1}^{2}\right)^{\frac{3}{2}}}-\frac{1}{8}\right) {y_1}.
  \end{array}
\end{equation}
 \begin{equation}
  \begin{array}{ll}
f_{13}=&{m_2} \left(\frac{1}{\left(\left({x_1} +1\right)^{2}+{y_1}^{2}\right)^{\frac{3}{2}}}-\frac{1}{\left(\left({x_3} +1\right)^{2}+{y_3}^{2}\right)^{\frac{3}{2}}}\right) \left(-\left({x_3} +1\right) {y_1} +{y_3} \left({x_1} +1\right)\right)+\\
& {m_4} \left(\frac{1}{\left(\left({x_1} -1\right)^{2}+{y_1}^{2}\right)^{\frac{3}{2}}}-\frac{1}{\left(\left({x_3} -1\right)^{2}+{y_3}^{2}\right)^{\frac{3}{2}}}\right) \left(\left(-{x_3} +1\right) {y_1} +{y_3} \left({x_1} -1\right).\right)
  \end{array}
\end{equation}

 \begin{equation}
  \begin{array}{ll}
f_{14}=&-2 {m_2} \left(\frac{1}{\left(\left({x_1} +1\right)^{2}+{y_1}^{2}\right)^{\frac{3}{2}}}-\frac{1}{8}\right) {y_1} + \\
& {m_3} \left(\frac{1}{\left(\left({x_1} -{x_3} \right)^{2}+\left({y_1} -{y_3} \right)^{2}\right)^{\frac{3}{2}}}-\frac{1}{\left(\left({x_3} -1\right)^{2}+{y3}^{2}\right)^{\frac{3}{2}}}\right) \left(-\left(-{x_3} +1\right) {y_1} -{y_3} \left({x_1} -1\right)\right).
  \end{array}
\end{equation}

 \begin{equation}
  \begin{array}{ll}
  f_{23}=& {m_1} \left(\frac{1}{\left(\left({x_1} +1\right)^{2}+{y_1}^{2}\right)^{\frac{3}{2}}}-\frac{1}{\left(\left({x_1} -{x_3} \right)^{2}+\left({y_1} -{y_3} \right)^{2}\right)^{\frac{3}{2}}}\right) \left(\left({x_3} +1\right) {y_1} -{y_3} \left({x_1} +1\right)\right)-\\
  &2 {m_4} \left(\frac{1}{8}-\frac{1}{\left(\left({x_3} -1\right)^{2}+{y_3}^{2}\right)^{\frac{3}{2}}}\right) {y_3}.
  \end{array}
\end{equation}

 \begin{equation}
  \begin{array}{ll}
  f_{24}=& 2 {m_1} \left(\frac{1}{\left(\left({x_1} +1\right)^{2}+{y_1}^{2}\right)^{\frac{3}{2}}}-\frac{1}{\left(\left({x_1} -1\right)^{2}+{y_1}^{2}\right)^{\frac{3}{2}}}\right) {y_1} +\\
  &2 {m_3} \left(\frac{1}{\left(\left({x_3} +1\right)^{2}+{y_3}^{2}\right)^{\frac{3}{2}}}-\frac{1}{\left(\left({x_3} -1\right)^{2}+{y_3}^{2}\right)^{\frac{3}{2}}}\right) {y_3}.
  \end{array}
\end{equation}

 \begin{equation}
  \begin{array}{ll}
  f_{34}=& {m_1} \left(\frac{1}{\left(\left({x_1} -{x_3} \right)^{2}+\left({y_1} -{y_3} \right)^{2}\right)^{\frac{3}{2}}}-\frac{1}{\left(\left({x_1} -1\right)^{2}+{y_1}^{2}\right)^{\frac{3}{2}}}\right) \left(\left(-{x_3} +1\right) {y_1} +{y_3} \left({x_1} -1\right)\right)-\\
  & 2 {m_2} \left(\frac{1}{\left(\left({x_3} +1\right)^{2}+{y_3}^{2}\right)^{\frac{3}{2}}}-\frac{1}{8}\right) {y_3}.
  \end{array}
\end{equation}

Note that the equations $f_{ij}(x,m)=0$ are homogeneous and in fact linear with respect to $m$. We can label the body with the largest mass value to be $m_2$ and the corresponding mass is normalized to one. Therefore $m_2=1$ and $0<m_i\leq 1$ for $i=1, 3, 4$. Then the problem of convex central configurations for the planar $4$-body problem is to find the zeros of $f_{ij}=0$ in the bounded domain $\mathbf{U}$ with bounded mass parameters $m_1, m_3, $ and $m_4$. 

It is easy to use Matlab command \lq\lq fmincon\rq\rq\,\, or \lq\lq fsolve\rq\rq\,\,  to numerically find the convex central configuration in the domain $\mathbf{U}$ for any given positive masses. But the commands \lq\lq fmincon\rq\rq\, or \lq\lq fsolve\rq\rq\, can not give any information whether the solution is unique even locally. We will use the interval analysis and the  Krawczyk operator to produce such information for some given mass values $m$. 
\begin{theorem}\label{MainTheorem}
For each given mass vector  $m_0=[m_{10}, m_{20}, m_{30},m_{40}]\in (\mathbf{R}^+)^4$  at a meshed grid point  such that $m_{20}=1$ and $0.6\leq m_{i0}\leq 1$ for $i=1,3,4$, there exists a unique convex central configuration in the order as given in Figure \ref{CC3}.
\end{theorem}
The computer-assisted proof of this theorem relies on the interval analysis which we now turn to, and the algorithm will be given at the end of the next section. The precise meaning for "a meshed grid point" in the statement and the final proof of this Theorem \ref{MainTheorem} will be given in Step 2 of the proof of the main Theorem \ref{MainMainTheorem} in section \ref{sec5}. 

\begin{remark} For the convenience of implementing the interval arithmetic computation, we need to cover $\mathbf{U}$ by intervals. 
Let $$\underline{\mathbf{U}}=[0,1]\times[0.267,1.733]\times [0,1]\times[-1.733, -0.267]$$
and $$\bar{\mathbf{U}}=[-1,1]\times[0,1.733]\times [-1,1]\times[-1.733, 0].$$
Then $$\underline{\mathbf{U}}\subset \mathbf{U} \subset \bar{\mathbf{U}}.$$
For a given mass vector $m_0$ in Theorem \ref{MainTheorem}, we find a unique convex central configuration $x_0 \in \underline{\mathbf{U}}$ by using the Krawczyk operator. Extending the uniqueness of $m_0$ to a neighborhood $B(m_0,\epsilon)$ in the corresponding position  neighborhood $B(x_0, r)$, we check that any $x\in \bar{\mathbf{U}}\backslash B(x_0,r)$ is not a central configuration for any $m\in B(m_0,\epsilon).$ As a result, for any $m\in B(m_0,\epsilon)$, there exists a unique convex central configuration. 
 \end{remark}

\section{Interval Analysis and Krawczyk Operator}
 \subsection{Notations and Concepts of Interval Arithmetic}
 
 In order to determine whether the equations of central configuration $f_{ij}(x,m)=0$ have any zeros, we employ interval arithmetic to test whether the functions $f_{ij}(x,m)$ contain zeros within specific intervals. Although the concept of interval arithmetic/analysis has a long history, the development of algorithms and computer programs to implement it dates back to the 1960s. A very nice survey of interval arithmetic was given by Alefeld and Mayer \cite{AM} which we will follow as a reference to provide a brief introduction to this topic. For more information, readers are referred to \cite{AM} and \cite{MooreR}.
 
 Let $[a]=[\underline{a},\bar{a}]$ and $[b]=[\underline{b},\bar{b}]$ be real closed intervals. 
If $\underline{a}=\bar{a}$, i.e., if $[a]$ consists only of the number $a$, then we identify the interval $[a,a]$ with the real number $a$, i.e., $a=[a,a].$   Let $\circ$ be one of the basic arithmetic operations $\{+,-,\cdot, /\}.$ The corresponding operations for intervals $[a]$ and $[b]$ are defined by 
 \begin{equation}\label{Rule1}
 [a]\circ [b]=\{ a\circ b| a\in[a], b\in [b]\},
 \end{equation}
 where we assume $0\notin [b]$ in case of division.
 
One of the advantages of interval computations is the fact that $[a]\circ [b]$ can be represented by using only the bounds of $[a]$ and $[b]$. For example:
$$[a]+[b]=[\underline{a}+\underline{b}, \bar{a} +\bar{b}], \hspace{1cm} [a]-[b]=[\underline{a}-\bar{b}, \bar{a} -\underline{b}],$$
 $$[a]\cdot[b]=[\min\{ \underline{ab},\underline{a}\bar{b},\bar{a}\underline{b},\bar{a}\bar{b}\}, \max\{ \underline{ab},\underline{a}\bar{b},\bar{a}\underline{b},\bar{a}\bar{b}\}].$$
 If we define $\frac{1}{[b]}=\left\{ \frac{1}{b}|b\in[b]\right\}$ for $0\notin [b]$, then $[a]/[b]=[a]\cdot \frac{1}{[b]}.$  
 
 Standard interval functions $\varphi \in \Psi=\{ \sin,$ $ \cos$, $\tan,$ $\arctan$, $\hbox{exp}, $ $\hbox{ln}$, $\hbox{abs},$ $\hbox{sqr},\hbox{sqrt}\}$ are defined via their range  
 \begin{equation}\label{Rule2}
 \varphi([x])=\{\varphi(x)|x\in[x]\}.
 \end{equation} 

 Suppose that $f: D \subseteq \mathbb{R} \rightarrow \mathbb{R}$ is defined by a mathematical expression $f(x)$ involving a finite number of elementary operations $+, -, \cdot, /$ and standard functions $\varphi\in\Psi$. If we replace the variable $x$ by an interval $[x] \subseteq D$ and evaluate the resulting expression using the rules in \eqref{Rule1} and \eqref{Rule2}, the output is again an interval, denoted by $f([x])$. This is commonly referred to as the interval arithmetic evaluation of $f$ over $[x]$. We assume that $f([x])$ exists whenever it is used in this paper, without explicitly stating so each time.
 
Observed that 
 $$[x]\subseteq [y] \Rightarrow f([x])\subseteq f([y]) \hbox{ and } x\in [x] \Rightarrow f(x)\in f([x]),$$
 it is easy to prove the following useful property 
 $$R(f;[x])\subseteq f([x]),$$ 
 where $R(f;[x])$ denotes the range of $f$ over $[x]$.
 
 The above fundamental property of interval arithmetic states that we can calculate lower and upper bounds for the range of a function over an interval using only the bounds of the interval itself. However, when we use the interval arithmetic evaluation $f([x])$, we may overestimate this range. The accuracy of the approximation to the range of $f$ over an interval $[x]$ strongly depends on how the expression for $f(x)$ is formulated. To illustrate these differences, we construct the Example \ref{example2} to compare the interval arithmetic evaluation $f([x])$ with the range $R(f;[x])$ in different expressions, and to demonstrate that the false statement that $f(x)$ contains zeros can arise due to the overestimation of $f([x])$.

 \begin{example}\label{example2} Let $f(x)$ be the rational function $$f(x)=\frac{2x^2}{x^2+9}-1.$$ 
\begin{itemize}
    \item  Observed that $f^\prime (x)=\frac{36x}{(x^2+9)^2}$,  $f(x)$ is increasing in the interval $[x]=[1,4]$. Therefore 
    $R(f,[1, 4]) = [-0.8, 0.28]$. It is easy to see that the equation $f(x)=0$ has a unique solution $x=3$ in this interval.
    \item By direct computations, $$f([1,4])=\frac{2\cdot [1,4]^2}{[1,4]^2+9}-1= \frac{2\cdot [1,16]}{[1,16]+9}-1 =\frac{[2,32]}{[10,25]}-1=[2,32]\cdot \frac{1}{[10,25]}-1$$
    $$ =[2,32]\cdot \left[\frac{1}{25},\frac{1}{10}\right]-1 =\left[\frac{2}{25},\frac{32}{10}\right]-1=[-0.92,2.2]\supseteq [-0.8,0.28],$$
     which confirms that  $R(f;[x])\subseteq f([x]).$ On the other hand, for $x\not=0$ we can rewrite $f(x)$ as 
    $$f(x)=\frac{2}{1+9/x^2}-1.$$
     $$f([1,4])=\frac{2}{1+9/[1,4]^2}-1=\frac{2}{1+9/[1,16]}-1=\frac{2}{1+[9/16,9]}-1$$
     $$=\frac{2}{[25/16,10]}-1=[2/10, 32/25]-1=[-8/10,7/25]=[-0.8, 0.28],$$
     which implies that $f([1,4])=R(f,[1,4]).$ This shows that the interval arithmetic evaluation depends on its expressions. 
     \item Now let us consider two sub-intervals   $[x_1]=[1,5/2]$, $[x_2]=[5/2,4]$.
 By simple computations we have
 $$f([x_1])=f([1,5/2])=\frac{2\cdot [1,5/2]^2}{[1,5/2]^2+9}-1=\frac{[2,25/2]}{[10,61/4]}-1=[-53/61,1/4].$$
 $$f([x_2])=f([5/2,4])=\frac{2\cdot [5/2,4]^2}{[5/2,4]^2+9}-1=\frac{[25/2,32]}{[61/4,25]}-1=[-1/2,67/61].$$
 The equation $f(x)=0$ may have solutions in both $[x_1]$ and $[x_2]$. But we can exclude the interval $[x_1]$ by splitting it into two smaller intervals $[x_{11}]=[1,2]$ and $[x_{12}]=[2,5/2]$. 
 $$ f([x_{11}])=f([1,2])=\left[-\frac{11}{13},-\frac{1}{5}\right],\hspace{1cm} f([x_{12}])=f([2,5/2])=\left[-\frac{29}{61},-\frac{1}{26}\right],$$
 which shows that $f(x)=0$ has no solutions in the intervals $[x_{11}]$ and $[x_{12}].$ Therefore there is no solutions in $[x_1]=[x_{11}]\bigcup [x_{12}].$ 
\end{itemize}
 \end{example}

Note that our functions $f_{ij}$ only involve basic arithmetic operations and square root. We employ interval arithmetic evaluation to determine if $f_{ij}$ contains zeros in given intervals. Example \ref{example2} demonstrates how we can rule out the existence of zeros in $[x]$, as well as an example of a false statement of zeros in $[x]$ due to $R(f;[x])\subseteq f([x])$. 

Interval arithmetic has been implemented on many platforms such as INTLAB and SageMath. INTLAB is a MATLAB library for interval arithmetic routines and it is written in MATLAB. SageMath is a free open-source mathematics software systems. The implementation of our computations on interval arithmetic is on SageMath 9.6 (published in 2022). Now let us introduce the Krawczyk operator which is an extension of Newton's method to search zeros of nonlinear system by incorporating interval computations.

 \subsection{Krawczyk Operator}
 
 Let $F: \mathbf{R}^n\rightarrow \mathbf{R}^n$ be a $C^1$-map. The Krawczyk operator (Krawczyk 1969;  Neumeier  1990; Alefeld 1994) is an interval analysis tool to establish the existence and uniqueness of zero for a system of $n$ nonlinear equations in the same number $n$ of variables: 
\begin{equation}
F(x)=0.
\end{equation}
The method proposed by Krawczyk for finding zeros of $F$ is as follows. Set 
\begin{itemize}
\item $[x]\subset \mathbf{R}^n$ be an interval set (i.e., Cartesian product of intervals).
\item $x_0 \in [x].$ Typically $x_0$ is chosen to be the midpoint of $[x]$. We will denote this by $x_0=mid([x])$.
\item $C\in \mathbf{R}^{n\times n}$ be a linear isomorphism. 
\end{itemize}
The Krawczyk operator is defined to be 
\begin{equation}
K(x_0, [x], F):=x_0-C\cdot F(x_0)+(I -C\cdot dF([x])) ([x]-x_0).
\end{equation}
The motivation and concise derivation of Krawczyk operator can be found in \cite{MZ}, where the operator was used to study the existence of central configurations for planar $n$-body problem with equal masses for $n=5,6,7$. More detailed discussion is referred to \cite{MooreR}.

\begin{lemma}\label{KO}
\begin{enumerate}
\item If $x^*\in  [x] $ and $F(x^*)=0$, then $x^* \in K(x_0, [x], F). $
\item If  $K(x_0, [x], F) \subset int[x],$ then there exists in $[x]$ exactly one solution of equation $F(x)=0$. This solution is non-degenerate, i.e., $dF(x)$ is an isomorphism. 
\item If $K(x_0, [x], F) \cap [x] =\emptyset$, then $F(x)\not= 0$ for all $x\in [x].$
\end{enumerate}
\end{lemma}

\begin{remark}
\begin{itemize}
\item For the concrete Krawczyk operator used in our problem the parameter $m$ is not considered as an interval, and we employ IFT to cover the neighboring mass points. The referee suggests that treating the masses parameter as an interval in the Krawczyk operator may be more effective than the IFT-based approach. Currently, it is unclear to us whether this would indeed be the case. However, we find that the proof using IFT is conceptually more appealing.

\item In the Krawczyk operator, $C$ is usually taken to be the inverse of $dF(x_0)$ which is updated in each iteration. If $dF(x^*)$ is singular, we can not use the Krawczyk operator. 
\item The point 2 in the above lemma gives us the way to establish the existence of unique zero of $F$ in $[x]$, whilst the point 3 rules out the existence of zeros in $[x]$. There is a quite common case that $[x]\subseteq K(x_0,[x],F)$ when the diameter of the interval $[x]$ is not small enough. In this case Krawczyk operator won't give us some useful information. In our program, we just simply split it into two small subintervals by bisecting the longest side of the interval.

\item We use Sagemath 9.6 to implement the interval computations and set RIF= RealIntervalField( ).
\end{itemize}
\end{remark}
The mass vectors are formulated by fixing $m_2=1$ and taking $m_i$ at $\delta_0+\frac{(1-\delta_0) k}{M_1}$ for $k=0,\cdots,$ $ M_1$ and $i=1,3,4, \delta_0=0.6$. Our program run through all the cases.  We are able to prove that there is a unique convex central configuration for each given mass vector. 
We first use the interval arithmetic computation and the Krawczyk operator to prove Theorem \ref{MainTheorem}. 

\subsection{The Algorithm to Prove Theorem \ref{MainTheorem}}
Take $m_0=[m_{01},1,m_{03},m_{04}]$ and $0< m_{0i}\leq 1$ at the meshed grid points at $\delta_0+\frac{(1-\delta_0) k}{M_1}$.  The notation $f_{ij}(x)=f_{ij}(x,m_0)$ is used for each fixed $m_0$.

The algorithm consists of four steps:

{\bf Step 1.} For the convenience of interval arithmetic computation, we use the interval $\underline{\mathbf{U}}$ such that  $[x_1,y_1,x_3,y_3] \in \underline{\mathbf{U}}$ for the admissible set $\mathbf{U}$ , i.e.  
$$x_1\in [0,1], y_1\in [0.267, 1.733], x_3\in [0,1], \hbox{ and } y_3\in [-1.733,-0.267]. $$
 We split the configuration space into small intervals. Each side is divided into $N_1$ equal intervals. The typical initial interval is $[x]_k=[x_{1}]_{k_1}\times [y_{1}]_{k_2}\times [x_{3}]_{k_3}\times [y_{3}]_{k_4}$ where $[x_1]_{k_1}=[\frac{k_1}{N_1},\frac{k_1+1}{N_1}]$ refers to the $k_1$-th subinterval for $x_1$ and the others are similar.  
 
{\bf Step 2.} The initial interval $[x]_k$ is ruled out if one of the six functions $f_{ij}([x]_k)$ doesn't contain zero. The rest intervals are candidates for the existence of solutions and they are labeled as $[x]_k$, $k=0,1,2,\cdots, k_{end}$. If no intervals are left, then there is no solution for the given mass in the given configuration space and we stop. Otherwise, go to next step.  

{\bf Step 3.}  
 The Krawczyk operator is used as a part of iteration process for each interval $[x]_k$ from Step 2: 
 
 Let $F(x)=[f_{12}(x), f_{13}(x), f_{14}(x),f_{23}(x)].$
\begin{enumerate} 
    \item given $[x]_k\subseteq \underline{\mathbf{U}}\subset \mathbf{R}^4$
    \item Compute $[y]=K(mid[x]_k,[x]_k,F)$
    \item if $[y]\subset int[x]_k$ and the diameter $d[y]=\bar{y}-\underline{y}$ of $[y]$ is smaller than the error tolerance, then return  {\bf success and $[y]$ and go to Step 4}
    
    elseif $[y]\subset int[x]_k$ and the diameter $d[y]=\bar{y}-\underline{y}$ of $[y]$ is not smaller than the error tolerance, then {\bf go to 2 by replacing $[x]_k$ by $[y]$}
    
    elseif $[x]_k\subset [y]$, then {\bf bisect the longest side of the interval $[x]_k$ and go to   Step 2 for the two smaller intervals}
    
    elseif $[x]_k\cap [y]=\emptyset$, then return {\bf no solution in $[x]_k$}
    
    elseif $[x]_k\cap [y]\not=\emptyset$, then {\bf set $[x]_{k+1}:=[y]\cap [x]_k$ and go to Step 2}
\end{enumerate}

{\bf Step 4.} If Step 3 returns {\bf success and $[y]$}, $f_{24}([y])$ and $f_{34}([y])$ are evaluated. If both contain zero, then $[y]$ is a central configuration for the given mass vector $m_0$ and it returns a {unique solution $[y]$}. 

In SageMath 9.6, there are two styles for printing intervals: ‘brackets’ style and ‘question’ style.
In question style, we print the “known correct” part of the number, followed by a question mark. The question mark indicates that the preceding digit is possibly wrong by $\pm 1.$ In the report of the computation results, we may use either of them.

If only one solution is obtained at the end of the program, it implies the existence of a unique central configuration in the given order as in Figure \ref{CC3}. Two examples are provided to demonstrate the numerically rigorous proof of the existence and uniqueness of the convex central configurations for a given mass.

\begin{example}
For $m_0=[1,1,1,1]$, there is a unique convex central configuration and 
$x_1=-1.?e-20,$ $y1=1.00000000000000000000?$, $x_3=-1.?e-20$, \\$y_3=-1.00000000000000000000?.$
\end{example}
In this case, we take $N_1=20$. Each side of the configuration is partitioned into 20 smaller intervals with length about $0.05$ to $0.0733$. There are $20^4=160,000$ intervals in total to check in Step 2. Most of them are ruled out because at least one of the six functions $f_{ij}$ contains no zero. Among them only 52 intervals are left for Krawczyk operator in Step 3. Then only one of these intervals 
returns the above unique solution and others return no solution. We also note that no bisection is needed for these intervals. So we conclude that there is a unique central configurations for this mass $m_0$ in the given ordering. 
{This way we recover the well known result (\cite{A96}): The square is the only convex central configuration in the planar $4$-body problem with equal masses.}

\begin{example}\label{ex4} For $m_0=[0.2, 1, 0.3, 0.4]$, there is a unique convex central configuration and $$x_1\in [0.15385328707521665441880634609065, 0.15385328707521665448042670172252],$$
$$y_1\in [1.4086619698548151890849667722929, 1.4086619698548151891412057380232],$$ 
$$x_3\in [0.11611158428853550145374041262273, 0.11611158428853550155068635233120],$$ 
$$y_3\in [-1.484878202646704369219144317714, -1.484878202646704369133519439332].$$
\end{example}
In this case, we take $N_1=15$. Each side of the configuration is partitioned into 15 smaller intervals with length about $0.0667$ to $0.0977$. There are totally $15^4=50,625$ intervals to check in Step 2.  Only 62 intervals are left for Krawczyk operator in Step 3. But this time, 60 of them return no solutions. Two of them 
are bisected their intervals and went back to Step 2.  Only one of them returns the above unique solution and others return no solution. We still have the conclusion that there is a unique central configurations for this mass $m_0$ in the given ordering. 

This completes the proof of Theorem \ref{MainTheorem} for the given mass $m_0$ at the meshed grid points. 
To get the proof for all masses, we have to fill in the gaps among the discrete masses by the Cartesian product of intervals. Here we employ implicit function theorem to produce such domains. Note that for any central configuration $x_0=x_0(m_0)=[x_{10},y_{10},x_{30}, y_{30}]$ with the corresponding mass $m_0$, $F(x_0,m_0)=0$.  By implicit function theorem, if a nondegeneracy condition on $(x_0,m_0)$ is fulfilled, there exist two radius $r>0$ and $\epsilon>0$, and a unique function $x=x(m)$ in the neighborhood $B(x_0,r)$ of $x_0$ and $B(m_0, \epsilon)$ of $m_0$ such that $x_0=x(m_0)$ and $F(x(m),m)\equiv 0$. Since $m_2=1$ is fixed once and for all, $m_0$ means $m_0=(m_{10}, m_{30}, m_{40})$ when $m_0$ is considered as a variable.  Hence $$B(m_0,\epsilon)=\{(m_1,m_3,m_4)| \sqrt{(m_1-m_{10})^2+(m_3-m_{30})^2+(m_4-m_{40})^2}<\epsilon\}.$$ 
 In the next section, we will refine the implicit function theorem to estimate the radius $r$ and $\epsilon$ of the neighborhoods of the concerning point. For each such mass neighborhood, we prove the existence and uniqueness of convex central configuration. 

\section{Implicit Function Theorem: existence, uniqueness and size of valid domain}\label{sec4}

Before we introduce the implicit function theorem (IFT) and its applications to our case, we first briefly go over the notations and the definitions that are used in this paper following \cite{JCB}, especially the norm we used should be consistent. For $x=(x_1,x_2,\cdots, x_n) \in \mathbb{R}^n$,  $\|x\|=\sqrt{x_1^2+x_2^2+\cdots +x_n^2}$ is the standard Euclidean norm. For a given real number $r>0$ and $x_0\in \mathbb{R}^n,$
$$B(x_0,r)=\{x\in\mathbb{R}^n | \|x-x_0\|<r\}$$
denotes the open ball of radius $r$ centered at $x_0.$  For given positive integers $n_1, n_2 \in \mathbb{N}$, the norm $\|A\|$ of a linear map $A:\mathbb{R}^{n_1}\rightarrow \mathbb{R}^{n_2}$ is defined by
$$\|A\|:=\sup\{\|Ax\| |x\in B(0,1)\subseteq \mathbb{R}^{n_1}\},$$
which is equivalent to $$\|A\|=\sup \{\|Ax\| |x\in \mathbb{R}^{n_1} \hbox{ and } \|x\|=1\} \hbox{ or } \|A\|=\sup \left\{\frac{\|Ax\|}{\|x\|} |x\in \mathbb{R}^{n_1}\right\}$$ 
with $\|Ax\|$ the norm of $Ax$ in $\mathbb{R}^{n_2}$. The standard Euclidean norm is easily computed in Matlab by using the command \lq norm\rq.

For our later purpose, let 
\begin{eqnarray*}
   f:\mathbb{R}^{n_1}\times \mathbb{R}^{n_2} &\rightarrow& \mathbb{R}^{n_1} \\
   (x,m) &\mapsto& f(x,m)
\end{eqnarray*} 
be a multivariable vector-valued function. $f$ has the form $f=(f_1,f_2,\cdots, f_{n_1})$ with components $f_1, f_2,\cdots, f_{n_1}:\mathbb{R}^{n_1}\times \mathbb{R}^{n_2} \rightarrow \mathbb{R}$ multivariable real-valued functions of $(x,m)$. 
For our problem $f(x,m)$ is a function of position $x$ with mass parameter $m$. $D_xf(x_0,m_0)$ and $D_mf(x_0,m_0)$ are Jacobian matrices of $f$ with respect to $x$ and $m$ respectively.

The second derivatives of $f$ with respect to $x$ is an array of $n_1$ Hessian matrices one for each component $f_j$:  
$$D_{x}^2 f=[H_x(f_1), H_x(f_2), \cdots, H_x(f_{n_1})],$$
whose norm $\|D_{x}^2 f(x_0,m_0)\|$ is computed as
$$\|D_{x}^2 f(x_0,m_0)\|= \sup\{ \| D_x^2 f(x_0,m_0) (v,w)\| | v\in \mathbb{R}^{n_1}, w\in \mathbb{R}^{n_1} \hbox{ and } \|v\|=\|w\|=1\}. $$
The second derivative $D_xD_mf(x_0,m_0)$ of $f$ with respect to variables $m$ and then variable $x$ is defined similarly.

Now it is ready to state IFT. 

\subsection{Implicit Function Theorem}
We start with a useful lemma.

\begin{lemma}\label{cont}
Given a continuous map $\varphi: \mathrm{cl}\mathrm{B}(y_{0},r)\times U \rightarrow \mathbb{R}^{n_{1}}$ with $\mathrm{cl}\mathrm{B}(y_{0},r)\subset\mathbb{R}^{n_{1}}$ a closed ball and $U \subset\mathbb{R}^{n_{2}}$ an open subset and a fixed point $x_0\in U$, the function 
\begin{eqnarray*}
    l:U &\rightarrow& \mathbb{R} \\
      x &\mapsto& l(x):=\max_{y\in \mathrm{cl}\mathrm{B}(y_{0},r)}\Vert\varphi(y,x)-\varphi(y,x_{0})\Vert
\end{eqnarray*} is continuous on $U$.
\end{lemma}
\begin{proof} By triangle inequality, for any $x_{1}\in U$
\begin{eqnarray*}
\|\varphi(y,x)-\varphi(y,x_0)\| &\leq& \|\varphi(y,x)-\varphi(y,x_1)\|+\|\varphi(y,x_1)-\varphi(y,x_0)\|\\
&\leq & \max_{y\in \mathrm{cl}\mathrm{B}(y_{0},r)}\|\varphi(y,x)-\varphi(y,x_1)\|+\max_{y\in \mathrm{cl}\mathrm{B}(y_{0},r)}\|\varphi(y,x_1)-\varphi(y,x_0)\|\\
&\leq & \max_{y\in \mathrm{cl}\mathrm{B}(y_{0},r)}\|\varphi(y,x)-\varphi(y,x_1)\|+l(x_1).
\end{eqnarray*}
So $$l(x)=\max_{y\in \mathrm{cl}\mathrm{B}(y_{0},r)}\|\varphi(y,x)-\varphi(y,x_0)\|\leq  \max_{y\in \mathrm{cl}\mathrm{B}(y_{0},r)}\|\varphi(y,x)-\varphi(y,x_1)\|+l(x_1),$$
that is
$$l(x)-l(x_1)\leq \max_{y\in \mathrm{cl}\mathrm{B}(y_{0},r)}\|\varphi(y,x)-\varphi(y,x_1)\|.$$
Similarly, we have
$$l(x_1)-l(x)\leq \max_{y\in \mathrm{cl}\mathrm{B}(y_{0},r)}\|\varphi(y,x_1)-\varphi(y,x)\|.$$ 
So $\forall x_{1}\in U$, we have \[\vert l(x)-l(x_{1})\vert\leq\max_{y\in\mathrm{cl}\mathrm{B}(y_{0},r)}\Vert\varphi(y,x)-\varphi(y,x_{1})\Vert.\] 
Take a closed ball $\mathrm{cl}\mathrm{B}(x_{1},r_{x_{1}})\subset U$, then $\varphi$ is uniformly continuous on $\mathrm{cl}\mathrm{B}(y_{0},r)\times\mathrm{cl}\mathrm{B}(x_{1},r_{x_{1}})$. Thus, $\forall\epsilon>0$, $\exists\delta>0$, if $\Vert(y_{3},x_{3})-(y_{2},x_{2})\Vert<\delta$, then $\Vert\varphi(y_{3},x_{3})-\varphi(y_{2},x_{2})\Vert<\epsilon$. So for $x\in U$ such that $\Vert x-x_{1}\Vert<\delta$, we have $\Vert(y,x)-(y,x_{1})\Vert<\delta$, $\Vert\varphi(y,x)-\varphi(y,x_{1})\Vert<\epsilon$ for any $y\in \mathrm{cl}\mathrm{B}(y_{0},r)$ which means that 
$$|l(x)-l(x_1)|<\epsilon.$$
The proof is complete.
\end{proof}

Now we use the lemma to prove a version of implicit function theorem used on the estimation of the valid domain.


\begin{theorem}
Let $Y\subset\mathbb{R}^{n_{1}}$, $X\subset\mathbb{R}^{n_{2}}$ be nonempty open sets. Let $Y\times X\owns(y,x)\mapsto f(y,x)\in\mathbb{R}^{n_{1}}$ be a $\mathscr{C}^{1}$ map. For $(y_{0},x_{0})\in Y\times X$, suppose $D_{y}f(y_{0},x_{0})$ is invertible. Then for any neighborhood $\mathscr{O}(y_{0})$ of $y_{0}$, there is a neighborhood $\mathscr{O}(x_{0}) $ of $x_{0}$ and a unique continuous map $g:\mathscr{O}(x_{0})\rightarrow\mathscr{O}(y_{0})$ such that

1. $g(x_{0})=y_{0}$,

2. $f(g(x),x)=f(y_{0},x_{0})$ for all $x\in\mathscr{O}(x_{0})$,

3. $g$ is continuous at point $x_{0}$.
\end{theorem}
\begin{proof} {The proof is based on Banach Fixed-Point Theorem.}

Let $T:=D_{y}f(y_{0},x_{0})$ and $\omega_{0}=f(y_{0},x_{0})$. Define 
\begin{eqnarray*}
   \varphi: Y\times X &\rightarrow & \mathbb{R}^{n_{1}}\\
    (y,x)&\mapsto& \varphi(y,x):=y-T^{-1}(f(y,x)-\omega_{0}).
\end{eqnarray*} 
Then $\varphi(y,x)=y$ iff $f(y,x)=\omega_{0}$. We will show that $\varphi(\cdot,x)$ has a fixed point. From the definition of $\varphi$, we have \[D_{y}\varphi(y,x)=I-T^{-1}D_{y}f(y,x).\] 
Since {$D_{y}\varphi(y,x)$} is continuous and $D_{y}\varphi(y_{0},x_{0})=0$, there exists  $r>0$ such that { $\Vert D_{y}\varphi(y,x)\Vert\leq \frac{1}{2}$ for any $y\in \text{cl}\text{B}(y_{0},r)$ and $x\in\mathrm{B}(x_{0},\epsilon_{1})$.} By the Mean Value Inequality, 
\begin{align*}
\Vert\varphi(y,x_{0})-y_{0}\Vert &=\Vert\varphi(y,x_{0})-\varphi(y_{0},x_{0})\Vert\\
& \leq\sup_{0\leq t\leq 1}\Vert D_{y}\varphi((1-t)y_{0}+ty,x_{0})\Vert\Vert y-y_{0}\Vert\\
&\leq\frac{1}{2}\Vert y-y_{0}\Vert\leq\frac{r}{2}, \forall y\in\mathrm{cl}\mathrm{B}(y_{0},r).\\
\end{align*}
For each $x$, let $l(x):=\max_{y\in \text{cl}\text{B}(y_{0},r)}\Vert\varphi(y,x)-\varphi(y,x_{0})\Vert$ as in the previous Lemma \ref{cont}. Since $l(x)$ is continuous by the lemma and $l(x_{0})=0$, there exists {$\epsilon\in (0,\epsilon_{1})$} such that\[l(x)<\frac{r}{2},\,\, \forall x\in\text{B}(x_{0},\epsilon).\] For each $x\in\text{B}(x_{0},\epsilon)$, we have \[\Vert\varphi(y,x)-y_{0}\Vert\leq\Vert\varphi(y,x)-\varphi(y,x_{0})\Vert+\Vert\varphi(y,x_{0})-y_{0}\Vert<\frac{r}{2}+\frac{r}{2}=r, \,\, \forall y\in\text{cl}\text{B}(y_{0},r).\] Thus, for each $x\in\text{B}(x_{0},\epsilon)$, $\varphi(\cdot,x)$ maps \[\text{cl}\text{B}(y_{0},r)\owns y\mapsto\varphi(y,x)\in\text{cl}\text{B}(y_{0},r).\] {For $y_{1}$, $y_{2}\in\mathrm{cl}\mathrm{B}(y_{0},r)$ and $x\in\mathrm{B}(x_{0},\epsilon)$,\[\Vert\varphi(y_{1},x)-\varphi(y_{2},x)\Vert\leq\max_{0\leq t\leq 1}\Vert D_{y}\varphi((1-t)y_{1}+ty_{2},x)\Vert\Vert y_{1}-y_{2}\Vert\leq\frac{1}{2}\Vert y_{1}-y_{2}\Vert.\]} 
Hence by {Banach} Fixed-Point Theorem, for each $x\in\text{B}(x_{0},\epsilon)$, there exists a $g(x)\in\text{cl}\text{B}(y_{0},r)$ such that $\varphi(g(x),x)=g(x)$ or equivalently $f(g(x),x)=\omega_{0}$. 
\end{proof}

Based on the proof of the implicit function theorem, one can get the following estimate about the size of the valid domain.

\begin{theorem}(\cite{JCB}, Proposition 3.12)\label{Imp1}
Let $Y\subset\mathbb{R}^{n_{1}}$ and $X\subset\mathbb{R}^{n_{2}}$ be two nonempty open sets. Let $f:Y\times X\rightarrow\mathbb{R}^{n_{1}}$ be a $\mathscr{C}^{2}$ map. For some $(y_{0},x_{0})\in Y\times X$, $D_{y}f(y_{0},x_{0})$ is invertible. Define\[L_{x}:=\Vert D_{x}f(y_{0},x_{0})\Vert ,\ M_{y}:=\left\Vert(D_{y}f(y_{0},x_{0}))^{-1}\right\Vert\] and for any given $R_{1}>0$, $R_{2}>0$, define
\begin{align*}
K_{yy}&:=\sup_{y\in\mathrm{B}(y_{0},R_{1})}\left\Vert D_{y}^{2}f(y,x_{0})\right\Vert, \\
K_{xx}&:=\sup_{(y,x)\in\mathrm{B}((y_{0},x_{0}),R_{2})}\left\Vert D_{x}^{2}f(y,x)\right\Vert, \\
K_{xy}&:=\sup_{(y,x)\in\mathrm{B}((y_{0},x_{0}),R_{2})}\left\Vert D_{x}D_{y}f(y,x)\right\Vert. \\
\end{align*}
Then for all $r\leq\min\{P,R_{2}\}$ with $P:=\min\left\{\frac{1}{2M_{y}K_{yy}},R_{1}\right\}$, and {$\epsilon\leq \min\left\{\sqrt{R_{2}^{2}-r^2},\frac{1}{4M_{y}K_{xy}}\right\}$}
 satisfying \[M_{y}(L_{x}+K_{xy}r+K_{xx}\epsilon)\epsilon<\frac{r}{2},\] $\forall x\in\mathrm{B}(x_{0},\epsilon)$ there exists a {unique $g(x)\in\mathrm{cl}\mathrm{B}(y_{0},r)$} such that $f(g(x),x)=f(y_{0},x_{0})$.
\end{theorem}
\begin{proof}
Using the Fundamental Theorem of Calculus, we have 
\begin{align*}
D_{y}\varphi(y,x_{0})&=\int_{0}^{1}D_{y}^{2}\varphi((1-s)y_{0}+sy,x_{0})(y-y_{0})\mathrm{d}s\\
&=-T^{-1}\int_{0}^{1}D_{y}^{2}f((1-s)y_{0}+sy,x_{0})(y-y_{0})\mathrm{d}s.\\
\end{align*}
Thus\[\Vert D_{y}\varphi(y,x_{0})\Vert\leq M_{y}K_{yy}\Vert y-y_{0}\Vert,\ \forall y\in\mathrm{B}(y_{0},R_{1}),\] by noting $T=D_{y}f(y_{0},x_{0})$ as in the previous proof.
From the Mean Value Inequality, 
\begin{align*}
\Vert\varphi(y,x_{0})-y_{0}\Vert &=\Vert\varphi(y,x_{0})-\varphi(y_{0},x_{0})\Vert\\
&\leq\max_{0\leq t\leq 1}\Vert D_{y}\varphi((1-t)y_{0}+ty,x_{0})\Vert\Vert y-y_{0}\Vert\\
&\leq M_{y}K_{yy}\Vert y-y_{0}\Vert\Vert y-y_{0}\Vert\\
&\leq M_{y}K_{yy}P\Vert y-y_{0}\Vert\\
&\leq\frac{1}{2}\Vert y-y_{0}\Vert, \forall y\in\mathrm{B}(y_{0},P),\\
\end{align*}
where \[P:=\min\left\{\frac{1}{2M_{y}K_{yy}},R_{1}\right\}.\] Using again the Fundamental Theorem of Calculus we get
\begin{align*}
D_{x}\varphi(y,x)&=D_{x}\varphi(y,x_{0})+\int_{0}^{1}D_{x}^{2}\varphi(y,(1-s)x_{0}+sx)(x-x_{0})\mathrm{d}s\\
&=D_{x}\varphi(y_{0},x_{0})+\int_{0}^{1}D_{y}D_{x}\varphi((1-s)y_{0}+sy,x_{0})(y-y_{0})\mathrm{d}s\\
&+\int_{0}^{1}D_{x}^{2}\varphi(y,(1-s)x_{0}+sx)(x-x_{0})\mathrm{d}s.\\
\end{align*}
By triangle inequality and submultiplicative property of induced norms we have
\begin{align*}
\Vert D_{x}\varphi(y,x)\Vert&\leq \Vert D_{x}\varphi(y_{0},x_{0})\Vert+\int_{0}^{1}\Vert D_{y}D_{x}\varphi((1-s)y_{0}+sy,x_{0})\Vert\Vert(y-y_{0})\Vert\mathrm{d}s\\
&+\int_{0}^{1}\Vert D_{x}^{2}\varphi(y,(1-s)x_{0}+sx)\Vert\Vert(x-x_{0})\Vert\mathrm{d}s\\
&\leq M_{y}(L_{x}+K_{xy}\Vert y-y_{0}\Vert+K_{xx}\Vert x-x_{0}\Vert),\ \forall(y,x)\in\mathrm{B}((y_{0},x_{0}),R_{2}).\\
\end{align*}
From the Mean Value Inequality, 
\begin{align*}
\Vert\varphi(y,x)-\varphi(y,x_{0})\Vert&\leq\max_{0\leq t\leq 1}\Vert D_{x}\varphi(y,(1-t)x_{0}+tx)\Vert\Vert x-x_{0}\Vert\\
&\leq M_{y}(L_{x}+K_{xy}\Vert y-y_{0}\Vert+K_{xx}\Vert x-x_{0}\Vert)(\Vert x-x_{0}\Vert),\ (y,x)\in\mathrm{B}((y_{0},x_{0}),R_{2}).\\
\end{align*}
For a given $r\leq\min\{P,R_{2}\}$, we need the inequality \[\Vert\varphi(y,x)-\varphi(y,x_{0})\Vert\leq\frac{r}{2}\] hold for all {$y\in\mathrm{cl}\mathrm{B}(y_{0},r)$}, and all $x\in\mathrm{B}(x_{0},\epsilon)$. Requring\[M_{y}(L_{x}+K_{xy}r+K_{xx}\epsilon)\epsilon\leq\frac{r}{2}\] will suffice. One can solve the inequality above to get a respective {$$\epsilon \leq \min\left\{ \sqrt{R_{2}^{2}-r^2},\frac{1}{4M_{y}K_{xy}}\right\}$$} for a given $r\leq\min\{P,R_{2}\}$.{
\begin{align*}
D_{y}\varphi(y,x)&=D_{y}\varphi(y,x_{0})+\int_{0}^{1}D_{x}D_{y}\varphi(y,(1-s)x_{0}+sx)(x-x_{0})\mathrm{d}s\\
&=D_{y}\varphi(y_{0},x_{0})+\int_{0}^{1}D_{y}^{2}\varphi((1-s)y_{0}+sy,x_{0})(y-y_{0})\mathrm{d}s\\
&+\int_{0}^{1}D_{x}D_{y}\varphi(y,(1-s)x_{0}+sx)(x-x_{0})\mathrm{d}s,\ \forall y\in\mathrm{cl}\mathrm{B}(y_{0},r).
\end{align*}
Noting $D_{y}\varphi(y_{0},x_{0})=0$, we have
\begin{align*}
\Vert D_{y}\varphi(y,x)\Vert&\leq \int_{0}^{1}\Vert D_{y}^{2}\varphi((1-s)y_{0}+sy,x_{0})\Vert\Vert(y-y_{0})\Vert\mathrm{d}s\\
&+\int_{0}^{1}\Vert D_{x}D_{y}\varphi(y,(1-s)x_{0}+sx)\Vert\Vert(x-x_{0})\Vert\mathrm{d}s\\
&\leq M_{y}K_{yy}\Vert y-y_{0}\Vert+M_{y}K_{xy}\Vert x-x_{0}\Vert\\
&\leq\frac{3}{4},\ \forall y\in\mathrm{cl}\mathrm{B}(y_{0},r),\ \forall x\in\mathrm{B}(x_{0},\epsilon).
\end{align*}
From the Mean Value Inequality,\[\Vert\varphi(y_{1},x)-\varphi(y_{2},x)\Vert\leq\max_{0\leq t\leq 1}\Vert D_{y}\varphi((1-t)y_{1}+ty_{2},x)\Vert\Vert y_{1}-y_{2}\Vert\leq\frac{3}{4}\Vert y_{1}-y_{2}\Vert,\]for $y_{1}$, $y_{2}\in\mathrm{cl}\mathrm{B}(y_{0},r)$ and $x\in\mathrm{B}(x_{0},\epsilon)$.
}
\end{proof}

Combining Theorem \ref{Imp1} and our setting about the central configurations, we have immediately 

\begin{theorem}\label{Thm4.1}
Let $F=[F_1, F_2,F_3,F_4]=[f_{12}, f_{13},f_{14}, f_{23}]$. Then $F$ is a map from $\mathbf{U}\times(\mathbb{R^+})^3\owns (x,m)$ to $\mathbb{R}^4$ where the admissible region $\mathbf{U}\subset\mathbb{R}^4$ is defined in Theorem \ref{Theorem1}. Since $m_2$ is fixed to be $1$  as  the largest value once and for all,  the mass vector $m\in (\mathbb{R^+}^3)$ is then $m=[m_1,m_3,m_4]\in (0,1]^3$. Let $x_0\in \mathbf{U}$ be a central configuration for $m_0$ such that $F(x_0,m_0)=0$ and $D_xF(x_0,m_0)$ is nonsingular. Define
$$L_m:=\| D_mF(x_0,m_0)\|,  \hspace{0.5cm} M_x:=\| (D_xF (x_0,m_0))^{-1}\| $$
and for given $R_1>0, R_2 >0$, set 
$$K_{xx}:= sup \{\|D_x^2 F(x,m_0)\| | x\in B(x_0,R_1)\},$$
$$K_{mm}:= sup \{\|D_x^2 F(x,m)\| | (x, m) \in B((x_0,m_0),R_2)\},$$ and
$$K_{xm}:= sup \{\|D_xD_m F(x,m)\| | (x, m) \in B((x_0,m_0),R_2)\},$$
Set $P: =\min \left\{ \frac{1}{2 M_x K_{xx}}, R_1 \right\}.$ Then for all $r\leq \min\{P,R_2\}$ and $\epsilon\leq \min\{\sqrt{R_2^2-r^2},\frac{1}{4M_{x}K_{xm}}\}$ satisfying
\begin{equation}\label{ineq1}
M_x (L_m +K_{xm} r+K_{mm} \epsilon) \epsilon \leq r/2,
\end{equation}
$F(x,m)=0$ has a unique solution on $B(x_0, r)$ for any given $m\in B(m_0, \epsilon)$. 
\end{theorem}
\begin{remark}
  In our computation program, in order to have a large existence and uniqueness neighborhood about $(x_0,m_0)$, we    choose $r=\min\{P, R_2\}$ and $\epsilon$ be the maximum value satisfying the inequality \eqref{ineq1} or we take
  \begin{equation}
      \epsilon=\min\left\{\frac{-M_x(L_m+K_{xm}r)+\sqrt{(M_x(L_m+K_{xm}r))^2+2M_xK_{mm}r}}{2M_xK_{mm}},
      \sqrt{R_2^2-r^2},\frac{1}{4M_{x}K_{xm}}\right\}.
  \end{equation} These two values strongly depend on $R_1$ and $R_2$.
\end{remark}
We give an example to illustrate our computation program to show that there exists a unique convex central configuration for any masses in the corresponding $\epsilon$-neighborhood. 
\begin{example}\label{ex5}
Let $x_0$ be the unique convex central configuration for mass $m_0$  with $m_0=[0.2,1,$ $0.3,$ $0.4]$ and $x_0=[x_{10},y_{10},x_{30},y_{30}]$ as given in Example \ref{ex4}. To compute $r$ and $\epsilon$, we take 
$$x_{10}=0.1538532870752166, 
y_{10}=1.4086619698548151,  $$ $$
x_{30}=0.1161115842885355,
y_{30}=-1.4848782026467043.$$ 
Let $R_1=0.1$ and $R_2=0.2$. By direct computation, $M_x=3.193848,$ $L_m=0.591292$, $K_{xx}=1.880567$, $K_{xm}=1.9330632$, $K_{mm}=2.160289$. Then $r= 0.083247$ and  $\epsilon=0.016540.$
Theorem \ref{Thm4.1} affirms that there exists a unique solution  $x=x(m)$ in $x\in B(x_0,r)$ for $m\in B(m_0,\epsilon).$ 

If we choose $R_1=0.2$ and $R_2=0.2$, then we have $r=0.066053,$ and $\epsilon=0.013632.$ This illustrates that the values of $r$ and $\epsilon$  depend on the choices of $R_1$ and $R_2$. In our computations, we take $R_1=0.1$ and $R_2=0.2$ to have a larger $\epsilon$.

{
To complete the proof of the uniqueness of the convex central configuration for any $m\in B(m_0,\epsilon)$, it is necessary to prove that any $x\in \bar{\mathbf{U}}\backslash B(x_0,r)$ is not a solution of the central configuration equations for this $m$. We have to change the balls to intervals in order to use interval arithmetic. Let $$\mathcal{I}(m_0,\epsilon)=[m_{10}-\epsilon, m_{10}+\epsilon]\times [m_{30}-\epsilon, m_{30}+\epsilon]\times[m_{40}-\epsilon, m_{40}+\epsilon]$$ be a mass interval centered at $m_0$ with radius $\epsilon$.  Let $$\mathcal{I}(x_0,r)=[x_{10}-r, x_{10}+r]\times [y_{10}-r, y_{10}+r]\times [x_{30}-r, x_{30}+r]\times [y_{30}-r, y_{30}+r]$$ be a position interval centered at $x_0$ with radius $r$. Then $$\mathcal{I}\left(x_0, \frac{r}{2}\right)\subseteq B(x_0,r)$$ and $$\mathcal{I}\left(m_0, \frac{\epsilon}{\sqrt{3}}\right)\subseteq B(m_0,\epsilon)\subseteq \mathcal{I}(m_0,\epsilon).$$ It is easy to prove the following lemma and its proof is omitted.
\begin{lemma}\label{lemmaU}
 Suppose there exists a unique solution  $x=x(m)$ in $x\in B(x_0,r)$ for $m\in B(m_0,\epsilon).$ If any $x\in \bar{\mathbf{U}}\backslash \mathcal{I}(x_0,\frac{r}{2})$ is not a solution for any $m \in \mathcal{I}(m_0,\epsilon)$, then there exists a unique solution $x=x(m)$ in $\mathbf{U}$ for $m \in B(m_0,\epsilon)$. If any $x\in \bar{\mathbf{U}}\backslash \mathcal{I}(x_0,\frac{r}{2})$ is not a solution for any $m \in \mathcal{I}(m_0,\frac{\epsilon}{\sqrt{3}})$, then there exists a unique solution $x=x(m)$ in $\mathbf{U}$ for $m \in \mathcal{I}(m_0,\frac{\epsilon}{\sqrt{3}})$.
\end{lemma} 
To show that any $x\in \bar{\mathbf{U}} \backslash\mathcal{I}(x_0,\frac{r}{2})$ is not a solution for any $m \in \mathcal{I}(m_0,\epsilon)$, we first divide the interval $\bar{\mathbf{U}}\backslash\mathcal{I}(x_0,\frac{r}{2})$ into smaller ones $[x]$ and then we check that at least one of $f_{ij}([x],\mathcal{I}(m_0,\epsilon))$ doesn't contain zeros. We need to make sure that the union of $[x]$ covers the whole region out of the position ball. This shows that the mass ball $B(m_0,\epsilon)$ is a uniqueness mass ball. 
In Example \ref{ex5}, $\mathcal{I}(m_0,\epsilon)= [0.183460297345220, 0.216539702654780] \times $ \\
$[0.283460297345220, 0.316539702654780]\times [0.383460297345220, 0.416539702654780]$ \\
and $\mathcal{I}_{x_0}=[0, 0.112229923080813]\times$  $[0.268, 1.733]\times $ $[0, 1]\times $ $[-1.733, -0.268]$ is one of the eight intervals in the closure of $\bar{\mathbf{U}}\backslash\mathcal{I}(x_0,\frac{r}{2})$. But the six functions $f_{ij}(\mathcal{I}_{x_0},\mathcal{I}(m_0,\epsilon)) $ all contain zeros. We split the position interval  $\mathcal{I}_{x_0}$ into two subintervals by bisecting the longest sides and then check whether $f_{ij}$ contains zeros on the subintervals. Repeat the process until it proves that at least one of $f_{ij}$ doesn't contain zeros in the subintervals. By this way, it confirms that $B(m_0,\epsilon)$ is a uniqueness mass ball in $\mathbf{U}$.}

Note that one may worry that there exists an $m$ in  $ B(m_0,\epsilon)$ near boundary whose solution $x$ is also near the boundary of $B(x_0,r)$. Then such $x$ falls in $\bar{\mathbf{U}}\backslash\mathcal{I}(x_0,\frac{r}{2})$. If such case exists, our method can not confirm the uniqueness of $B(m_0,r)$. But we can apply our method to a smaller mass intervals $\mathcal{I}(m_0,\frac{\epsilon}{N})\subseteq B(m_0, \epsilon) $ for an appropriate $N\geq\sqrt{3}$ and prove $B(m_0,\frac{\epsilon}{N})$ is a uniqueness mass ball.  Fortunately, this case doesn't occur in our program. All mass balls $B(m_0,\epsilon)$ in our program are uniqueness mass balls. To accelerate the numerical computations, we can produce the uniqueness mass interval $\mathcal{I}(m_0,\frac{\epsilon}{\sqrt{3}})$  to cover the mass space by checking that at least one of $f_{ij}([x],\mathcal{I}(m_0,\frac{\epsilon}{\sqrt{3}}))$ doesn't contain zeros for all small intervals $[x]$ out of the corresponding position ball.




\end{example}

\section{Proof of the Theorem \ref{MainMainTheorem}}\label{sec5}

To prove the uniqueness of convex central configurations in Theorem \ref{MainMainTheorem}, we  create finite many uniqueness mass balls $B(m_0, \epsilon)$ so that the union of such balls can cover the whole mass space $[0.6,1]^3$. Here we describe the procedures to produce the uniqueness mass balls.  

Step 1: Divide the mass space $[0.6,1]^3$ into a grid with equal side length $[0.6, 1]$ for $M_1$ segments, such that $m_{ik}=0.6+\frac{0.4k}{M_1}$ for $i=1,3,4$ and $k=0,1,2,...,M_1$. The value of $M_1$ is determined based on experimental results. We choose $M_1=40.$

Step 2: {For each $m_0=[m_{1{k_1}}, 1, m_{3{k_3}},m_{4{k_4}}]$ at a meshed grid point in Step 1,} applying the algorithm  (interval arithmetic computation and Krawczyk operator) in Section 3, we can prove there is a unique central configuration $x_0$. {Now the proof of Theorem \ref{MainTheorem} is complete.}

Step 3: For such central configuration $x_0$ for $m_0$, applying the implicit function theorem in Section 4, we can find a uniqueness mass ball $B(m_0, \epsilon)$, where $\epsilon$ is a small positive number such that the central configuration is unique for all masses in the ball $B(m_0, \epsilon)$ as guaranteed by Theorem \ref{Thm4.1} and Lemma \ref{lemmaU}. 

Step 4: Check whether the union of all the uniqueness balls covers the mass space $[0.6,1]^3$. This follows from the fact that the compact mass space $[0.6,1]^3$ can be covered by a finitely many uniqueness mass balls and the fact that we can refine our mesh points so that the uniqueness balls  overlap ensuring that any mass point is in at least one uniqueness ball. 

Therefore, by Step 2 and Step 3, there exists a unique central configuration for any mass point in the meshed space, and by Step 4, this uniqueness extends to the entire mass space $[0.6,1]^3$. Thus, Theorem \ref{MainMainTheorem} is proven.

Initially, set $M_1=5$. We find that the radius $\epsilon$ of the uniqueness mass ball  is between $0.0092119$ and $0.016393$ with an overall  average radius about $0.01258436$. If the minimum of the radius of the uniqueness mass balls is $\epsilon_0$, we can cover the whole mass space by choosing $M_1> \frac{4\sqrt{3}}{20\epsilon_0} $ since a uniqueness ball $B(m_0, \epsilon_0)$ contains an interval centered at $m_0$ with each side length $\frac{2\sqrt{3}\epsilon_0}{3}.$ If $\epsilon_0=0.009$, $M_1>38$. In our program, we use $M_1=40$ and totally $64000$ uniqueness mass balls are computed to cover the whole mass space $[0.6,1]^3$. {According to Lemma \ref{lemmaU}, our program confirms the uniqueness of the mass ball $B(m_0,\epsilon)$   for each mass $m_0$ at the mesh grid in Step 1.  }

\section{Conclusions and Remarks}\label{sec6}

By combining the interval analysis and implicit function theorem, we give a computer assisted proof about the uniqueness of convex central configuration for any given mass vector $m\in [0.6, 1]^3$ in a fixed order in the planar $4$-body problem. This closed domain can be enlarged and improved further without modifying our method and program. So the conjecture is hopefully true for all masses except for those near the mass vector with one or more component being zeros.

However we find that it becomes more and more time consuming when we approach the boundary of the mass space because $\epsilon$ becomes smaller and smaller. This suggests the potential singular behavior when we take the limit to the boundary of the mass space.

For example, let's consider the mass space $([0.1,1])^3$ and let $M_1=10$. We determine that the radius $\epsilon$ of the uniqueness mass ball lies between $0.0001436$ and $0.02125$, with an average radius of approximately $0.00998445$. By using the minimum radius $\epsilon_0=0.0001436$, we can cover the entire mass space with $M_1> 5428$ and a total number of uniqueness mass balls equal to approximately $1.6\times 10^{11}$. It will take a rather long time to finish the computation using the current program, and it seems unpractical without further refinement of our current program.

Reflecting further on our methods, we find the following mathematical problems may be interesting. (1)  Given a function $f$ on a closed domain $D\subset \mathbb{R}^k$, for each point $q\in D$, we can find an open ball $B(q,r_q)$ with radius depending on the point $q$ and some property of the function $f$ (e.g. the valid domain for implicit function theorem in the current paper). In such a way, we get an open covering of the compact set $D$, and in principle we can select a finite number of them to cover $D$. Now the question is whether one can find an optimal one as in sphere packing problem. Here we are deliberately vague to allow refinements and variations. (2) Even more interesting problem would be trying to refine and optimize the size of the valid domain for the implicit function theorem with or without the conditions about the second order derivatives of the mapping. There is very few literature concerning this fundamental topic, {and we leave to future research}.

Yet from this perspective we also notice the following interesting result establishing the conjecture
about the uniqueness of the convex central configurations for masses $m_1,m_3,m_4$ in the following form near zeros in a given order. 

\begin{lemma}(\cite{CCLM}, Theorem 2)
    Given $\mu_i>0 (i=1,3,4)$, there exists $\epsilon_0>0$ such that for $0<\epsilon <\epsilon_0,$ the four masses $m=[\epsilon\mu_1,$ $ 1, $ $\epsilon \mu_3, $ $\epsilon \mu_4]$ admit a unique convex central configuration for a given order. 
\end{lemma}


Now we cannot merge the statement into our framework because $\epsilon_0$
is theoretical and not given a concrete numerical value. Note also there are other possibilities to investigate for masses approaching zero. It seems still a challenging problem even numerically to verify the conjecture. 


The final remark is that even we can prove the conjecture with computer assistance, it is still desirable to have a conceptual/analytical proof.

{\bf Acknowledgements}

{We would like to thank the valuable suggestions from Professor Piotr Zgliczynski and the anonymous referee.}

Shanzhong Sun is partially supported by National Key R\&D Program
of China (2020YFA 0713300), NSFC (No.s 11771303, 12171327, 11911530092, 11871045).
Zhifu Xie is partially supported by Wright W. and Annie Rea Cross Endowment Funds at the University
of Southern Mississippi.




\end{document}